\documentclass[12pt]{amsart}
\usepackage{txfonts}
\usepackage{amsfonts}
\usepackage{amsfonts}
\usepackage{amssymb,latexsym}
\usepackage{cite}
\usepackage{enumerate}
\newtheorem{theorem}{Theorem}

\newtheorem{proposition}[theorem]{Proposition}
\newtheorem{lemma}[theorem]{Lemma}
\newtheorem{definition}[theorem]{Definition}

\newtheorem{assumption}[theorem]{Assumption}

\newcommand {\p}{\partial}

\numberwithin{equation}{section}
\numberwithin{theorem}{section}

\usepackage{titletoc}

\usepackage{hyperref}
\hypersetup{hypertex=true,colorlinks=true,linkcolor=blue,anchorcolor=blue,citecolor=blue}
\begin{document}
\title[%Green Function to Complex $k$-Hessian Operator\\ 
The homogeneous complex k-Hessian equation]{
%On the Green Function to Complex $k$-Hessian Operator in a Punctured Domain \\ 
The Dirichlet Problem of Homogeneous Complex $k$-Hessian Equation in a $(k-1)$-pseudoconvex domain with isolated singularity}
\author{Zhenghuan Gao}
\address{Department of Mathematics, Shanghai University, Shanghai, 200444, China}
\email{gzh@shu.edu.cn}
\author{Xinan Ma} 
\address{School of Mathematical Sciences, University of Science and Technology of China, Hefei 230026,
	Anhui Province, China}
\email{xinan@ustc.edu.cn}

\author{Dekai Zhang}  
\address{Department of Mathematics, Shanghai University, Shanghai, 200444, China}
\email{dkzhang@shu.edu.cn}

\begin{abstract}
%In this paper, we prove the existence and uniqueness for the Green function to complex $k$-Hessian operator in a punctured $(k-1)$-pseudoconvex domain. We also give uniform estimates for the Green function near the singularity.\\ 
In this paper, we consider the homogeneous  complex $k$-Hessian equation in $\Omega\setminus\{0\}$. 
 We prove the existence and uniqueness  of the  $C^{1,\alpha}$ solution by constructing approximating solutions.
 The key point for us is to construct the subsolution for approximating problem and establish uniform gradient estimates and complex Hessian estimates  which is independent of the approximation. 

\end{abstract}
\maketitle
%\tableofcontents 

\section{Introduction}
Let $\Omega$ be a smooth  bounded domain of $\mathbb C^n$ %. It is well known that the Laplace operator $\Delta $ has a unique Green function. for 
 and $k$ be an integer such that $1\leq k\leq n$. We consider the homogeneous complex $k$-Hessian equations
\begin{align*}
(dd^cu)^k\wedge \omega^{n-k}=0\quad\text{in }\Omega\setminus\{0\}.
\end{align*}
Let $u$ be a real $C^2$ function in $\mathbb C^n$ and $\lambda=(\lambda_1,\cdots,\lambda_n)$ be the eigenvalues of the complex Hessian $(\frac{\partial^2u}{\partial z_i\partial \bar z_j})$, the complex $k$-Hessian operator is defined by

\begin{align*}
	H_{k}[u]:=\sum\limits_{1\le i_1<\cdots i_k\le n}\lambda_{i_1}\cdots \lambda_{i_k},
\end{align*}
where $1\le k\le n$.
Using the operators $d=\partial+\overline\partial$ and $d^c=\sqrt{-1}(\overline \partial-\partial )$, such that $dd^c=2\sqrt{-1}\partial\overline \partial$, one gets
$$(dd^cu)^k\wedge\omega^{n-k}=4^nk!(n-k)!H_k[u]d\lambda,$$
where $\omega=dd^c|z|^2$ is the fundamental K\"{a}hler form and $d\lambda$ is the volume form. 
When $k=1$, $H_1[u]=\frac14\Delta u$. When $k=n$, $H_n[u]=\det u_{i\bar j}$ is the complex Monge-Amp\`ere operator.

\subsection{Some known results and motivations}
Let $S_k(D^2u)$ be the $k$-Hessian of a real $C^2$ function $u$ in $\mathbb R^n$. When $k>1$, the Hessian equations $S_k(D^2u)=f$ and $H_k[u]=f$ are both nonlinear. When $f>0$, the Hessian equation is nondegenerate. When $f$ vanishes somewhere, the Hessian equation is degenerate.

\subsubsection{Results on bounded domain}
For the Hessian equation on $\mathbb R^n$, its Dirichlet problem with positive $f$
\begin{align*}
\begin{cases}
S_k(D^2u)=f&\quad\text{in }\Omega,\\
u=\varphi&\quad\text{on }\partial\Omega,
\end{cases}
\end{align*}
was studied by Ivochkina \cite{Ivochkina1985} for $k=1,2,3,n$ on convex domain with further assumptions on $f$ and by Caffarelli-Nirenberg-Spruck \cite{CNS3} for general $f>0$ and $k=1,2,\cdots,n$ by assuming $\Omega$ is $(k-1)$-convex. B. Guan \cite{GuanBo1994}  showed the geometric condition on $\Omega$ could be removed by assumption of existence of a strict subsolution. 
%and Trudinger \cite{Trudinger1995} for $f>0$. 
In \cite{TrudingerWang1997}, Trudinger-Wang developed a Hessian measure theory for Hessian operator. One can see a survey in Wang \cite{WangXujia2009} for more related topics. For the complex $k$-Hessian equation in $\mathbb C^n$, Li \cite{LiSongying2004} solved its Dirichlet problem via the subsolution approach.

%Dirichlet problem for degenerate Hessian equation in $\mathbb R^n$ has been studied extensively.  One can see \cite{CNS1986degenerate,GTW1999,WangXujia1995,ITW2004,DongHongjie2006} for references. 
For Monge-Amp\`ere equation in a bounded domain of $\mathbb R^n$, when $f=0$, 
Caffarelli-Nirenberg-Spruck \cite{CNS1986degenerate} proved the $C^{1,1}$ regularity in a bounded convex domain.
For general $f\geq 0$, Guan-Trudinger-Wang \cite{GTW1999} proved the $C^{1,1}$ regularity result when  $f^{\frac{1}{n-1}}\in C^{1,1}$.  Due to the counterexample by Wang \cite{WangXujia1995}, $C^{1,1}$ regularity is optimal. For $k$-Hessian equation in $\mathbb R^n$, the $C^{1,1}$ regularity is obtained by Krylov \cite{Krylov1989,Krylov1994} and Ivochkina-Trudinger-Wang \cite{ITW2004}.

%and Ivochina-Trudinger-Wang\cite{ITW2004} 
%Dong \cite{DongHongjie2006} proved the $C^{1,1}$ regularity  for some  degenerate mixed type Hessian equations.
For complex Monge-Amp\`ere equation, Lempert \cite{Lempert1981,Lempert1983} proved the Dirichlet problem admits a smooth solution with a logarithm pole at the origin on a strictly convex punctured domain $\Omega\backslash\{0\}$ when $f=0$. As for strongly pseudoconvex domain, Guan \cite{GuanBo1998} and B\l ocki \cite{Blocki2000} proved the solution is  $C^{1,1}$. In \cite{GuanPengfei2002}, Guan obtained the $C^{1,1}$ regularity for the solution on a ring domain. For general $f\geq 0$, the optimal $C^{1,1}$ regularity was in Caffarelli-Kohn-Nirenberg-Spruck \cite{CKNS2},Krylov \cite{Krylov1989, Krylov1994} for strongly pseudoconvex domain.  %$f^{\frac{1}{k}}\in C^{1,1}$, some regularity results were obtained by 

\subsubsection{Results on unbounded domain}
The viscosity solution to nondegenerate $k$-Hessian equation on unbounded domain has been researched extensively.
%There are lots of results on the exterior Dirichlet problem for viscosity solutions of nondegenerate fully nonlinear equations. 
Caffarelli-Li \cite{CaffarelliLi2003} solved the viscosity solution to  the Monge-Amp\`ere equation $\mathrm{det }\, D^2u=1$ with prescribed asymptotic behavior at infinity. Bao-Li-Li \cite{BaoLiLi2014} studied  the $k$-Hessian equation case. For the related results on other type nondegenerate fully nonlinear equations, one can see \cite{BaoLi2013,LiLi2018,Li2019}.

In \cite{LiWang2015}, Li-Wang consider the $\mathrm{det}\, D^2u=0$ on a strip region $\Omega:=\mathrm R^n\times[0,1]$. By assuming two boundary functions are both strictly convex $C^{1,1}(\mathbb R^{n-1})$ functions, they obtained the solutions is $C^{1,1}(\overline\Omega)$. If the boundary functions are locally uniformly convex $C^{k+2,\alpha}(\mathbb R^{n-1})$ function, then $u$ is the unique $C^{k+2,\alpha}(\overline \Omega)$ function. %The global $C^{k+2,\alpha}$ regularity of the homogeneous Monge-Amp\`ere equation in a strip region was proved by Li-Wang [29] by assuming that the boundary functions are locally uniformly convex and $C^{k,\alpha}$. They showed that the uniform convexity of the boundary functions is necessary.

Recently, Xiao \cite{Xiaoling2022} and Ma-Zhang \cite{MaZhang2022} proved the $C^{1,1}$ regularity of Dirichlet  fot the homogeneous $k$-Hessian equation out of  $\Omega\subset\mathbb R^n$, by assuming $\Omega$ is starshaped, $(k-1)$-convex and and $1\leq k<\frac n2$ or $\Omega $ is $(k-1)$-convex and $1\leq k\leq n$ respectively.  For homogeneous complex $k$-Hessian equation, Gao-Ma-Zhang \cite{GMZ2022} obtained the $C^{1,1}$ regularity.

\subsubsection{Motivations}

Our paper is motivated by the research on the regularity of extremal function or Green function. In \cite{Klimek1985}, Klimek introduced the following extremal fucntion
$$g_\Omega(z,\xi)=\sup\{v\in \mathcal{PSH}(\Omega):v<0,\ v(z)\leq \log|z-\xi|+O(1)\}.$$
$g_\Omega(z,\xi)$ is also call the pluricomplex Green function on $\Omega$ with a logarithminc pole at $\xi$. If $\Omega$ is hyperconvex, Demailly \cite{Demailly1987} showed that $u(z)=g_\Omega(z,\xi)$ is continuous and is a unique solution to the homogeneous complex Monge-Amp\'ere equation,
\begin{align}\label{eqMA}
\begin{cases}
(dd^cu)^n=0&\quad\text{in }\Omega\setminus\{\xi\},\\
u=0&\quad\text{on }\partial\Omega,\\
u(z)=\log|z-\xi|+O(1)&\quad\text{as }z\rightarrow\xi.
\end{cases}
\end{align}
If $\Omega$ is strictly convex domain in $\mathbb C^n$ with smooth boundary, Lempert \cite{Lempert1981} proved \eqref{eqMA} admits a unique plurisubharmonic solution which is smooth. In the strongly pseudonconvex case, B. Guan \cite{GuanBo1998} proved $g_\Omega(z,\xi)\in C^{1,\alpha}(\overline\Omega\setminus\{\xi\})$ and later, B\l ocki  improved it to $C^{1,1}(\overline\Omega\setminus\{\xi\})$ in \cite{Blocki2000} and generalized it to several poles in \cite{Blocki2001}. Due to the counterexamples found by Bedford-Demailly \cite{BedfordDemailly1988}, $C^{1,1}$ regularity is optimal. 

P. Guan \cite{GuanPengfei2002} established  $C^{1,1}$ regularity of extremal function associated to intrinsic norms of Chen-Levine-Nirenberg \cite{CLN1969} and Beford-Taylor \cite{BedfordTaylor1979} by considering 
\begin{align*}
\begin{cases}
(dd^cu)^n=0&\quad\text{in }\Omega_0\setminus(\cup_{i=1}^m \Omega_i),\\
u=0&\quad\text{on }\partial\Omega_i, \ i=1,\cdots,n\\
u=1&\quad\text{on }\partial\Omega_0.
\end{cases}
\end{align*}
Applying the techniques from \cite{GuanPengfei2002}, B. Guan proved the $C^{1,1}$ regularity of pluricomplex Green function for the union of a finite collection of strongly pseudonconvex domains in $\mathbb C^n$.

In \cite{GMZRealinterior}, we considered the following homogeneous (real) $k$-Hessian equation in a punctured domain
\begin{equation}\label{GMZRealinterior}
\begin{cases}
S_k(D^2u)=0&\quad\text{in }\Omega\backslash\{0\},\\
u=c_k&\quad\text{on }\partial\Omega,\\
u(x)=h_k(x)&\quad\text{as }x\rightarrow 0
\end{cases}•
\end{equation}
where $c_k= 1$ and $h_k(x)=0$  if $k>\frac n2$, $c_k=-1$ and $h_k(x)=-|x|^{2-\frac nk}+O(1)$ if $k<\frac n2$, $c_k=0$ and $h_k(x)=\log|x|+O(1)$ if $k=\frac n2$. 
Assume that $\Omega$ is $(k-1)$-convex, we proved the existence  and  uniqueness  of $C^{1,1}$ solution to \eqref{GMZRealinterior}. Moreover the solution can be controlled pointwisely by fundamental solutions of  homogenous k-Hessian equations up to the second order. If $\Omega$ is also starshaped with respect to the origin, we proved the positive lower bound of the gradient of the solution and then we show a nearly monotonicity formula along the level set of the approximating solution.

\subsection{Our result}
In this section, we consider the following problem for complex $k$-Hessian equation
\begin{equation}\label{maineq} 
\begin{cases}
(dd^cu)^k\wedge \omega^{n-k}=0&\quad\text{in }\Omega\setminus\{0\},\\
u=-1&\quad\text{on }\partial\Omega,\\
u(z)=-|z|^{2-\frac{2n}k}+O(1)&\quad\text{as }z\rightarrow 0.
\end{cases}
\end{equation}

\begin{theorem}\label{mainthm}
Assume $1\leq k<n$. Let $\Omega$ be a smooth $(k-1)$-pseudoconvex domain containing the origin. Then there exists a unique $k$-subharmonic solution $u$ of \eqref{maineq} in $C^{1,\alpha}(\overline\Omega\setminus \{0\})$. Moreover, $u$ satisfies the estimate
\begin{align}
&-C\leq u+|z|^{2-\frac{2n}k}\leq 0,\label{mainest1}\\
&|Du|+|z||\Delta u|\leq C|z|^{1-\frac{2n}k}\label{mainest2}.
\end{align}
\end{theorem}

Here $k$-subharmonic function and $(k-1)$-pesudoconvex domain are introduced in Section \ref{sec2}. We suppose $\Omega$ contains the origin and we use the notation $\Omega_r=\Omega\setminus\overline B_r(0)$. We use $B_r$ instead of $B_r(0)$ for short. To prove Theorem \ref{mainthm}, we consider the approximating problem
\begin{equation}\label{eq::approx::sec1}
\begin{cases}
H_k[u^{\varepsilon,r}]=\varepsilon&\quad\text{in }\Omega_r,\\
u=\underline u&\quad\text{on } \partial B_r,
\end{cases}
\end{equation}
where $\underline u$ is a subsolution constructed in Section \ref{sec3}. The solution $u$ to \eqref{maineq} with be obtained by  approximating solution $u^{\varepsilon,r}$ to \eqref{eq::approx::sec1}.  The existence of $u^{\varepsilon,r}$ follows from subsolution method in \cite{LiSongying2004}. 

The rest of the paper is organized as follows. In Section \ref{sec2}, we first give the definition and some notations. Then we recall some new gradient estimates and complex Hessian estimates in \cite{GMZ2022} motivated by B. Guan \cite{GuanBo2007}, which will be used in the proof of \eqref{mainest2}. In Section \ref{sec3}, we establish uniform gradient estimates and complex Hessian estimates. Theorem \ref{mainthm} will be proved in the last section.

\section{Preliminaries}\label{sec2}

\subsection{Elementary symmetric functions}
For any $k=1,\cdots,n$ and $\lambda=(\lambda_1,\cdots,\lambda_n)\in\mathbb R^n$, the $k$-th elementary symmetric function on $\lambda$ is defined by 
$$S_k(\lambda):=\sum\limits_{1\le i_1< \cdots <i_k\le n}\lambda_{i_1}\cdots\lambda_{i_k}.$$
Let $S_k(\lambda|i)$ be the symmetric function with $\lambda_i=0$. Let $A=(a_{ij})\in\mathbb R^{n\times n}$ be an $n\times n$ matrix. Let $S_k(A)$ be the $k$-th elementary symmetric function on $A$, which is the sum of $k\times k$ principal minors of $A$. 
%defined by
%$$S_k(A):=\delta_{i_1\cdots i_k}^{j_1\cdots j_k}A_{i_1j_1}\cdots A_{i_kj_k},$$ where $\delta_{i_1\cdots i_k}^{j_1\cdots j_k}$ is the Kronecker symbol, which has the value $+1$ (respectively, $-1$) if $i_1,i_2\cdots i_k$
%are distinct and $(j_1j_2\cdots j_k)$ is an even permutation (respectively, an odd permutation) of
%$(i_1i_2\cdots i_k)$, and has the value 0 in any other cases. 
We use the convention that $S_0(A)=1$. It is clear that 
$S_k(A)=S_k(\lambda(A))$, where $\lambda(A)$ are the eigenvalues of $A$.

The elementary symmetric functions have  the following simple properties from \cite{Lieberman2005}.
\begin{align}
S_k(\lambda)=S_k(\lambda| i)+\lambda_i S_{k-1}(\lambda| i),
\end{align}
and 
\begin{align}
\sum_{i=1}^nS_{k}(\lambda| i)=(n-k)S_k(\lambda).
\end{align}
Recall the $\Gamma_k$-cone is defined by
\begin{align*}
	\Gamma_{k}:=\{\lambda\in \mathbb{R}^n\mid S_{i}(\lambda)>0, 1\le i\le k\}
\end{align*}
For $\lambda\in \Gamma_k$ and $1\leq l\leq k$, the well-known MacLaurin inequality (see \cite{Lieberman2005}) says
$$\bigg(\frac{S_k(\lambda)}{C_n^k}\bigg)^\frac1k\leq \bigg(\frac{S_l(\lambda)}{C_n^l}\bigg)^\frac1l.$$ 
%As a corollary, we have
%$$\sum_{i=1}^n\frac{\partial S_k^\frac1k(\lambda)}{\partial\lambda_i}\geq [C_n^k]^\frac1k.$$
%The following proposition holds due to the MacLaurin inequality and Garding's inequality, see \cite{WangXujia1994}.
%\begin{proposition}
%For $\lambda\in\Gamma_k$, we have
%\begin{align}\label{eq::prop2.1}
%\bigg(\prod_{i=1}^nS_{k-1}(\lambda|i)\bigg)^\frac1n\geq \frac kn \big(C_n^k\big)^\frac 1k \big(S_k(\lambda)\big)^\frac{k-1}k.
%\end{align}
%\end{proposition}
One can find the concavity property of $S_k^{\frac{1}{k}}$  in \cite{CNS3}.
\begin{proposition}\label{concavity}
$S_k^{\frac{1}{k}}$ is a concave function in $\Gamma_k$. %Moreover, $\log S_k$ is concave in $\Gamma_k$.
\end{proposition}

\subsection{$k$-subharmonic solutions}
In this section we give the definition of $k$-subharmonic functions and definition of $k$-pseudoconvex domains.
One can see the lecture notes by Wang \cite{WangXujia2009} for more properties of the $k$-Hessian operator, and see B\l ocki \cite{Blocki2005} for those of the complex $k$-Hessian operator.
We following the definition by B\l ocki \cite{Blocki2005} to give the definition of $k$-subharmonic functions.
\begin{definition}
Let $\alpha$ be a real $(1,1)$-form in $U$, a domain of $\mathbb C^n$. We say that $\alpha$ is $k$-positive in $U$ if the following inequalities hold
$$\alpha^j\wedge \omega^{n-j}\geq 0, \forall\,j=1,\cdots,k.$$
\end{definition}

\begin{definition}
	Let $U$ be a domain in $\mathbb C^n$. 
	
(1). A function $u:U\rightarrow \mathbb R\cup \{-\infty\}$ is called $k$-subharmonic if it is subharmonic and for all $k$-positive real $(1,1)$-form $\alpha_1,\cdots,\alpha_{k-1}$ in $U$,
$$dd^cu\wedge \alpha_1\wedge \cdots \wedge \alpha_{k-1}\wedge \omega^{n-k}\geq 0.$$
The class of all $k$-subharmonic functions in $U$ will be denoted by $\mathcal{SH}_k(U)$.

(2). A function $u\in C^2 (U)$ is called  \emph{$k$-subharmonic} (\emph{strictly $k$-subharmonic}) if $\lambda\big(\frac{\partial^2u}{\partial z_i\partial \bar z_j}\big)\in \overline \Gamma_{k}$ \ ($\lambda\big(\frac{\partial^2u}{\partial z_i\partial \bar z_j}\big)\in\Gamma_k$).
\end{definition}
If $u\in \mathcal{SH}_{k}(U)\cap C(U)$, $(dd^{c}u)^{k}\wedge \omega^{n-k}$ is well defined in pluripotential theory by B\l ocki \cite{Blocki2005}. We need the following comparison principle by B\l ocki\cite{Blocki2005} to prove the uniqueness of the continuous solution of the problem \eqref{maineq}.
\begin{lemma}\label{comparison0718}
Let $U$ be a bounded domain in $\mathbb C^n$,  $u, v\in \mathcal{SH}_k(U)\cap C(\overline U)$ satisfy
	\begin{align*}
		\begin{cases}
			(dd^cu)^k\wedge \omega^{n-k}\ge (dd^cv)^k\wedge \omega^{n-k}&\quad \text{in}\ U,\\
			u\le v &\quad\text{on}\ \p U.
		\end{cases}
	\end{align*}
Then $u\le v$ in $U$.
\end{lemma}
\subsection{Gradient estimates and complex Hessian estimates}
Motivated by \cite{GuanBo2007}, we proved the following new gradient estimates and complex Hessian estimates in \cite{GMZ2022}.
\begin{theorem}\label{thm::gradientestimate}
Let $u\in C^3(U)\cap C^1(\overline U)\cap \mathcal{SH}_k(U)$ be a  negative  solution to $H_k[u]=f$  in $U$,  where $f\in C^1(\overline U)$ is positive. Denote by 
$$P=|Du|^2(-u)^{-\frac{2n-k}{n-k}}.$$
Then
\begin{align}\label{eq::est::gradient}
\max_{\overline U}P\leq \max\bigg\{\max_{\partial U}P,\max_{\overline U}\bigg(\frac{2(n-k)}{k(2n-k)}\bigg)^2(-u)^{-\frac k{n-k}}|D\log f|^2\bigg\}
\end{align}
\end{theorem}

\begin{theorem}\label{thm::2ndest}
Let $u\in C^4(U)\cap C^2(\overline U)\cap \mathcal{SH}_k(U)$ be a negative solution to $H_k[u]=f$ in $U$, where $f\in C^2(\overline U)$ is positive. Assume that $P=|Du|^2(-u)^{-\frac{2n-k}{n-k}}$, $(-u)^{-\frac{k}{n-k}}|D\log f|^2$ and $(-u)^{-\frac{k}{n-k}}|D^2\log f|$ are bounded. Denote by 
$$H=u_{\xi\bar\xi}(-u)^{-\frac{n}{n-k}}(M-P)^{-\sigma},$$
where $M=2\max\limits_{\overline U}P+1$, $\sigma\leq \frac{n(n-k)}{8(2n-k)^2}$.
Then we have 
\begin{align}\label{eq::est::2nd}
\max_{\overline U}H\leq C+\max_{\partial U}H,
\end{align}
where $C$ is a positive constant depending only on $n$, $k$, $P$, $(-u)^{-\frac{k}{n-k}}|D\log f|^2$ and $(-u)^{-\frac{k}{n-k}}|D^2\log f|$.
\end{theorem}

We need the following  lemma by P. Guan\cite{GuanPengfei2002} to construct the subsolution of the complex $k$-Hessian equation in a ring domain.
\begin{lemma}\label{Guan2002lemma}
	Suppose that $U$ is a bounded smooth domain in $\mathbb{C}^n$. For $h, g\in C^m(U)$, $m\ge 2$, for all $\delta>0$, there is an $H\in C^m(U)$ such that
	
	\begin{enumerate}[(1)]
		\item
		$H\ge \max\{h, g\}$ \ \text{and}
		
		\begin{align*}
			H(z)=\left\{ {\begin{array}{*{20}c}
					{h(z), \quad   \text{if } \ h(z)-g(z)>\delta  }, \\
					g(z) , \ \quad \text{if } \ g(z)-h(z)>\delta;\\
			\end{array}} \right.
		\end{align*}
		\item
		{There exists}  $|t(z)|\le 1 $ {such that}
		
		\begin{align*}
			\left\{H_{i\bar j}(z)\right\}\ge
			\left\{\frac{1+t(z)}{2}g_{i\bar j}+\frac{1-t(z)}{2}h_{i\bar j}\right\},\ \text{for all} \ x\in\left\{|g-h|<\delta\right\}.
		\end{align*}
	\end{enumerate}
\end{lemma}
We can prove that $H$ is  $k$-subharmonic if $f$ and $g$ are both $k$-subharmonic by the concavity of $S^\frac1k$ in Proposition \ref{concavity}.

At the last of this subsection, we recall the definition of $(k-1)$-pseudoconvex domain.

\begin{definition}
A $C^2$ domain $U$ is called $(k-1)$-pseudoconvex if there is $C_U>0$, such that $\lambda(-d_{i\bar j}+C_U(d^2)_{i\bar j})\in\Gamma_{k}$ on $\partial U$, where $d(z)=\mathrm{dist}(z,\partial U)$ is the distance function from $z$ to $\partial U$.
%$\lambda(\mathcal L)\in \Gamma_{k-1}$ on $\partial U$, where $\mathcal L$ is a Levi form of $\partial U$, $\lambda(\mathcal L)$ are the eigenvalues of $\mathcal L$.
\end{definition}

\section{Solving the approximating problem in $\Omega\setminus B_r$}\label{sec3}
In this section, we will solve the approximating problem by  a-priori estimates and the subsolution method. Before this, we make an assumption on $\Omega$.
\begin{assumption}\label{assumption3.1}
Assume $\Omega$ contains the origin and $B_{r_0}\subset\subset\Omega\subset\subset B_{(1-\tau_0)R_0}$ for some $\tau_0\in(0,\frac12)$. 
\end{assumption}
Denote by $\Omega^\mu=\{z\in\Omega:d(z)<\mu\}$. In this section, we use $C$ and $c$ with subscript to denote some positve constant which are independent of $\varepsilon$ and $r$.

The following lemma about $(k-1)$-pesudoconvex domain in $\mathbb C^n$ is a parallel version to $(k-1)$-convex domain in $\mathbb R^n$ with can be found in \cite[Section 3]{CNS3}. It plays an important roles in constructing the subsolution.
\begin{lemma}
Let $\Omega$ be a smooth $(k-1)$-pseudoconvex bounded domain. There exists $\mu_0\in (0,\frac1{2C_\Omega})$ small enough such that $B_{r_0}\subset\subset\{z\in\Omega:d(z)>2\mu_0\}$. Moreover $\rho:=-d+C_\Omega d^2$ is smooth and strictly $k$-subharmonic and $H_k[\rho]\geq \epsilon_0$ in $\overline{\Omega^{2\mu_0}}$ for some $\epsilon_0>0$.%, where $\Omega_{2\mu_0}:=\{z\in\Omega:d(z)<2\mu_0\}$.
\end{lemma}

\subsection{The approximating equation.} We will approximate the solution to the homogeneous complex $k$-Hessian equation in $\Omega\backslash\{0\}$ by solutions to a sequence of nongenerate equation in $\Omega_r$. The existance of approximating solution can be obtained if we can construct a smooth subsolution. In the following, we use the technique from P. Guan \cite{GuanPengfei2002} %and the $(k-1)$-pseudoconvexity 
to construct a subsolution.

Denote $w:=-|z|^{2-\frac{2n}k}+R_0^{2-\frac{2n}k}-1+a_0\frac{|z|^2}{R_0^2}$, where %we choose 
$a_0=\frac12\Big((1-\tau_0)^{2-\frac{2n}k}-1\Big)R_0^{2-\frac{2n}k}$. Then by $\Omega\subset B_{(1-\tau_0)R_0}$, we have 
$$w\leq -\frac12\Big((1-\tau_0)^{2-\frac{2n}k}-1\Big)R_0^{2-\frac{2n}k}-1\quad\text{in } \overline \Omega.$$ 
By Proposition \ref{concavity}, we have
$$H_k^\frac1k[w]\geq H_k^\frac1k[-|z|^{2-\frac{2n}k}]+H_k^\frac1k[a_0\frac{|z|^2}{R_0^2}]= (C_n^k)^\frac1ka_0R_0^{-2}\quad\text{in } \Omega.$$
Then by Lemma \ref{Guan2002lemma}, we can construct a smooth and strictly $k$-subharmonic function $\underline u$ from $w$ and $\rho$.
\begin{lemma}
There is a strictly $k$-subharmonic function $\underline u\in C^\infty(\overline \Omega_r)$  satisfying 
\begin{align*}
\underline u(z)=\begin{cases}
K_0\rho(z)-1\quad\text{if}\quad d(z)\leq \frac{\mu_0}{M_0},\\
w(z)\quad\text{if}\quad d(z)>\mu_0,
\end{cases}
\end{align*}
\begin{align*}
\underline u(z)\geq \max\{K_0\rho(z)-1,w(z)\}\quad\text{if}\quad\frac{\mu_0}{M_0}\leq d(z)\leq \mu_0,
\end{align*}
\begin{align*}
H_k[\underline u]\geq \epsilon_1:=\min\{C_n^ka_0^kR_0^{-2k},K_0^k\epsilon_0\}\quad\text{in}\quad\Omega,
\end{align*}
where $K_0$ and $M_0$ are uniform constants.
\end{lemma} 
\begin{proof}
%By comparing $K_0\rho-1$ and $w$ on $\Omega^{2\mu_0}$, we find
Since $B_{r_0}\subset \{z\in\Omega:d(z)>2\mu_0\}$, by choosing $K_0=\frac{r_0^{2-\frac{2n}k}}{C_\Omega\mu_0^2-\mu_0}$, we find $\forall\, z\in \overline{\Omega^{2\mu_0}}\setminus\Omega^{\mu_0}$, there holds
\begin{align*}
w-(K_0\rho-1)=&-|z|^{2-\frac{2n}k}+R_0^{2-\frac{2n}k}+a_0\frac{|z|^2}{R_0^2}-K_0\rho\\
\geq& -r_0^{2-\frac{2n}k}+R_0^{2-\frac{2n}k}-K_0(-\mu_0+C_\Omega \mu_0^2)\\
\geq&R_0^{2-\frac{2n}k},
\end{align*}
%where the last inequality holds provided $K_0=\frac{r_0^{2-\frac{2n}k}}{C_\Omega\mu_0^2-\mu_0}$. 
For any $z\in\overline{\Omega_{\frac{\mu_0}{M_0}}}:=\{z\in\overline\Omega:d(z)\leq \frac{\mu_0}{M_0}\}$, there also holds
\begin{align*}
(K_0\rho-1)-w\geq \frac12\Big((1-\tau_0)^{2-\frac{2n}k}-1)\Big)R_0^{2-\frac{2n}k}+K_0\Big(-\frac{\mu_0}{M_0}+C_\Omega (\frac{\mu_0}{M_0})^2\Big)\geq\frac14\Big((1-\tau_0)^{2-\frac{2n}k}-1\Big)R_0^{2-\frac{2n}k},
\end{align*}
provided that $M_0$ is a positive solution to
\begin{align}\label{eq::0101}
K_0(-\frac{\mu_0}{M_0}+C_\Omega(\frac{\mu_0}{M_0})^2)\geq-\frac14((1-\tau_0)^{2-\frac{2n}k}-1)R_0^{2-\frac{2n}k}
\end{align}
In fact, we can choose $\tau_0$  small enough such that \eqref{eq::0101} holds if $M_0>1$.

%where we take a positive $M_0$ satisfying \begin{equation*}\label{eq::0923::1}K_0(-\frac{\mu_0}{M_0}+C_\Omega(\frac{\mu_0}{M_0})^2)=-\frac14((1-\tau_0)^{2-\frac{2n}k}-1)R_0^{2-\frac{2n}k}\end{equation*} to ensure the last equality. In fact, we can choose $\tau_0$  small enough such that quadratic equation admits a positive solution.

Take $\delta:=\min\{\frac14\Big((1-\tau_0)^{2-\frac{2n}k}-1\Big)R_0^{2-\frac{2n}k},R_0^{2-\frac{2n}k}\}$ and we apply Lemma \ref{Guan2002lemma} with $g=K_0\rho-1$, $h=w$ and $\delta$ on $\Omega^{2\mu_0}$, we obtain a smooth and strictly $k$-subharmonic function $\underline u$ in $\Omega^{2\mu_0}$. Moreover $\underline u=K_0\rho-1$ in $\Omega^{\frac{\mu_0}{M_0}}$, and $\underline u=w$ in $\overline{\Omega^{2\mu_0}}\setminus\Omega^{\mu_0}$. At last, we set $\underline u=w$ in $\Omega_r\setminus \Omega^{2\mu_0}$. By Lemma \ref{Guan2002lemma}, we have
$$H_k[\underline u]\geq \min\{H_k[w],H_k[K_0\rho]\}\geq \min\{C_n^ka_0^kR_0^{-2k},K_0^k\epsilon_0\}.$$
\end{proof}
We now consider the approximating equation
\begin{align}\label{approxeq}
\begin{cases}
H_k[u]=\varepsilon&\quad\text{in }\Omega_r,\\
u=\underline u&\quad\text{on }\partial\Omega_r
\end{cases}
\end{align}
Then $\underline u$ is a strictly subharmonic solution of above equation for any $\varepsilon<\epsilon_1$. By Li \cite{LiSongying2004}, \eqref{approxeq} admits a strictly $k$-subharmonic solution $u^{\varepsilon,r}\in C^\infty(\overline\Omega_r)$. 
Let $r_1= \min\{2^\frac{2k}{2k-n}R_0,\big(\frac{2a_0}{R_0^2}\big)^{-\frac{k}{2n}}\}$, $\forall\, r\leq r_1$, since $\underline u=-1$ on $\partial \Omega$ and $\underline u=-r^{2-\frac{2n}k}+R_0^{2-\frac{2n}k}-1+a_0\frac{r^2}{R_0^2}$ on $\partial B_r$, we have $\underline u\mid_{\partial \Omega_r}\leq -1$.%, $\forall\, r\leq r_1$, $r_1= \min\{2^\frac{2k}{2k-n}R_0,(\frac{2a_0}{R_0^2})^{-\frac{k}{2n}}\}$.  
By maximum principle, we have $u^{\varepsilon.r}\leq-1$ when $r\leq r_1$, $\varepsilon\leq \epsilon_1$. 

In the following, we want to derive a $(\varepsilon,r)$-independent uniform $C^2$ estimate for $u^{\varepsilon,r}$. We prove the following
\begin{theorem}
Suppose $\Omega$ be a smooth $(k-1)$-pseudoconvex bounded domain. Assume that $\Omega$ satisfies Assumption \ref{assumption3.1}.
For sufficient small $r>0$ and $\varepsilon>0$, \eqref{approxeq} admits a $k$-subharmonic solution $u^{\varepsilon,r}$, where $\underline u$ is constructed above.
Moreover, $u^{\varepsilon,r}$ satisfies the following estimates,
\begin{align}
-|z|^{2-\frac{2n}k}+R_0^{2-\frac{2n}k}-1+a_0\frac{|z|^2}{R_0^2}&\leq u^{\varepsilon,r}\leq -|z|^{2-\frac{2n}k}+r_0^{2-\frac{2n}k}-1,\label{eq::est::09051}\\
|Du^{\varepsilon,r}|&\leq C|z|^{1-\frac{2n}{k}},\\
|\partial\bar\partial u^{\varepsilon,r}|&\leq C|z|^{-\frac{2n}{k}},
\end{align}
where $C$ is a uniform positive constant which is independent of $\varepsilon$ and $r$.

In addition, if $\Omega$ is starshaped with respect to the origin, there is a uniform positive constant $c$ independent of $\varepsilon$ and $r$ such that
\begin{equation}\label{eq::0101::2}|Du^{\varepsilon,r}|\geq c_0|z|^{1-\frac{2n}k}.\end{equation}
\end{theorem}
%\begin{proof}
%By Evans-Krylov theory, there exists a solution of \eqref{} provided 
%The existence part follows by subsolution approach in Li \cite{LiSongying2004} since $H_k[\underline u]\geq 1$.
%\end{proof}

\subsection{$C^0$ estimate}
Since $\underline u$ is a subsolution to \eqref{approxeq}, we obtain that
\begin{align}\label{0825::eq::1}
u^{\varepsilon,r}\geq \underline u\quad\text{in }\Omega_r.
\end{align}

Let $$\overline u=-|z|^{2-\frac{2n}k}+r_0^{2-\frac{2n}k}-1.$$ By taking $r\leq \min\{r_1,r_2\}$, where $r_2=R_0(r_0^{2-\frac{2n}k}-R_0^{2-\frac{2n}k})^\frac12a_0^{-\frac12}$,
we have
$$u^{\varepsilon,r}\leq \overline u\quad\text{on }\partial \Omega_r.$$
%$u^{\varepsilon,r}=-1<\overline u$ on $\partial\Omega$ and $u^{\varepsilon,r}=\underline u\leq \overline u$ on $\partial B_r$ %provided $r\leq \min\{r_1,r_2\}$, $r_2=R_0(r_0^{2-\frac{2n}k}-R_0^{2-\frac{2n}k})^\frac12a_0^{-\frac12}$. 
Note that $H_k[u^{\varepsilon,r}]=\varepsilon>0=H_k[\overline u]$ in $\Omega_r$, it follows that
\begin{align}\label{0814::eq::3}
u^{\varepsilon,r}\leq \overline u\quad\text{in }\Omega_r.
\end{align}
By \eqref{0825::eq::1} and \eqref{0814::eq::3}, we obtain
\begin{align*}%\label{0825::eq::2}
-|z|^{2-\frac{2n}k}+R_0^{2-\frac{2n}k}-1+a_0\frac{|z|^2}{R_0^2}\leq u^{\varepsilon,r}\leq -|z|^{2-\frac{2n}k}+r_0^{2-\frac{2n}k}-1.
\end{align*}
This gives the $C^0$ estimate \eqref{eq::est::09051}.

\subsection{Gradient estimates} 
Base on the key estimate \eqref{eq::est::gradient}, we can prove the global gradient estimate in this subsection.
\subsubsection{\bf Reducing global gradient estimates to boundary gradient estimates.}
Since $u^{\varepsilon,r}<0$, $f=\varepsilon$, by Theorem \ref{thm::gradientestimate}, we have
$$\max_{\overline {\Omega_r}}P=\max_{\partial\Omega_r}P.$$
\subsubsection{\bf Boundary gradient estimates}
To prove boundary gradient estimates, we will construct barriers near $\partial\Omega$ and $\partial B_r$ respectively.

%Note that $\underline u\mid_{\partial B_r}<-1$, then
Since $u^{\varepsilon,r}=\underline u=-1$ on $\partial\Omega$ and $u^{\varepsilon,r}\geq \underline u$ in $\Omega_r$, we have
\begin{align}\label{0813::eq::1}
|Du^{\varepsilon,r}|=u^{\varepsilon,r}_\nu\leq \underline u_\nu\quad\text{on }\partial \Omega,
\end{align}
where $\nu$ is the unit outer normal to $\partial\Omega$. 
%Let $r_1=2^{2-\frac{2n}k}r_0$. 
Let $r_3\leq \min\{r_1,r_2,1\}$ and $h_1$ be the 
%Let $r_1=\min\{4^{\frac{k}{2k-2n}}R_0,\Big(2\big((1-\tau_0)^{2-\frac{2n}k}-1\big)\Big)^\frac{k}{2k-2n},1\}$,
%Consider a 
harmonic function $h_1$ in $\Omega\setminus B_{r_3}$ with $h_1=-1$ on $\partial\Omega$ and $h_1=-r_3^{2-\frac{2n}k}+r_0^{2-\frac{2n}k}-1$ on $\partial B_{r_3}$. Then we have $h_1\geq u^{\varepsilon,r}$ on $\partial\Omega_{r_3}$. So $h_1\geq u^{\varepsilon,r}$ in $\Omega_{r_3}$, and it follows that
\begin{align}\label{0813::eq::2}
u_\nu^{\varepsilon,r}\geq h_{1,\nu}>0\quad\text{on }\partial \Omega,
\end{align}  
That is there exist a positive constant $C$ such that 
\begin{align}\label{0825::eq::4}
0<C^{-1}\leq u^{\varepsilon,r}_\nu\leq C\quad\text{on }\partial \Omega.
\end{align}

Let $h_2$ be a harmonic function with $h_2=\underline u$ on $\partial B_r$ and $h_2=\overline u=-\frac12|z|^{2-\frac{2n}{k}}$%$h_2=\overline u=-|z|^{2-\frac{2n}{k}}+r_0^{2-\frac{2n}k}-1$ 
on $\partial B_{2r}$. Let
$$\tilde h_2(z)=r^{\frac{2n}k-2}(h_2(r z)+r^{2-\frac{2n}k})=r^{\frac{2n}k-2}h_2(r z)+1.$$
Then $\tilde h_2$ is a harmonic function in $B_2\setminus B_1$ with $\tilde h_2=a_0r^{\frac{2n}k}R_0^{-2}+r^{\frac{2n}k-2}(R_0^{2-\frac{2n}k}-1)$ on $\partial B_1$ and $\tilde h_2=-2^{1-\frac{2n}k}$%$\tilde h_2=-2^{2-\frac{2n}k}+1+r^{\frac{2n}k-2}(r_0^{2-\frac{2n}k}-1)$ 
on $\partial B_2$. Let
\begin{align}\label{eq::0101::3}\tilde u=r^{\frac{2n}k-2}u^{\varepsilon,r}(r z)+1,\end{align}
and
\begin{align}\label{eq::0101::4}\tilde {\underline u}=r^{\frac{2n}k-2}\underline u(r z)+1.\end{align}
By maximum principle, we have
%$$\underline u\leq u^{\varepsilon,r}\leq h_2\text{ in }B_{2r}\setminus B_r$$
%and
$$\tilde{\underline  u}\leq \tilde u\leq \tilde h_2\text{ in }B_2\setminus B_1.$$
Note that
%$$\underline u= u^{\varepsilon,r}= h_2=-r^{2-\frac{2n}k}+R_0^{2-\frac{2n}k}-1+a_0\frac{r^2}{R_0^2}\text{ on }\partial B_r,$$
%and
$$\tilde{\underline  u}= \tilde u= \tilde h_2=a_0r^{\frac{2n}k}R_0^{-2}+r^{\frac{2n}k-2}(R_0^{2-\frac{2n}k}-1)\text{ in }\partial B_1.$$
We obtain
$$D'\tilde u=D'\tilde{\underline u}=D'\tilde h_2=0\quad\text{on }\partial B_1,$$
and
\begin{align*}%\label{0825::eq::3}
0<c(n,k)\leq\tilde{\underline u}_\nu\leq \tilde u_\nu\leq \tilde h_{2,\nu}\leq \tilde C\quad\text{on }\partial B_1,
\end{align*}
where $\nu$ is the unit outer normal to $\partial B_1$, $\tilde C$ is independent of $r$ and $\varepsilon$.
%Hence 
%$$|D\tilde u|^2=|D'\tilde u|^2+(\tilde u_\nu)^2\leq |D\tilde{\underline u}|^2+|D\tilde h_2|\quad\text{on }\partial B_1.$$
So we  obtain
$$C^{-1}\leq |D\tilde u|\leq C\quad\text{on }\partial B_1.$$
By \eqref{eq::est::09051}, we have
\begin{align}\label{0825::eq::5}
|Du^{\varepsilon,r}|\leq Cr^{1-\frac{2n}k}=C|z|^{1-\frac{2n}k}\leq C(-u^{\varepsilon,r})^a\quad\text{on }\partial B_r,
\end{align}
and 
\begin{align}\label{eq::0922::1}
|Du^{\varepsilon,r}|\geq C^{-1}r^{1-\frac{2n}k}\quad\text{on }\partial B_r.
\end{align}
By \eqref{eq::est::gradient}, \eqref{eq::est::09051}, \eqref{0813::eq::1} and \eqref{0825::eq::5}, we obtain
$$|Du^{\varepsilon,r}|\leq C(-u^{\varepsilon,r})^a\leq C(|z|^{2-\frac{2n}k}-R_0^{2-\frac{2n}k}+1-a_0\frac{|z|^2}{R_0^2})\leq C|z|^{1-\frac{2n}k}\quad\text{in }\Omega_r.$$

\subsubsection{\bf Positive lower bound of $|Du^{\varepsilon,r}|$} 
Since $\partial\Omega$ is starshaped with respect to the origin, we have $t\cdot \nu>0$ on $\partial\Omega$, where $\nu$ is the unit outer normal to $\partial\Omega$, $t=(t_1,\cdots, t_{2n})=(y_1,\cdots, y_n, x_1,\cdots,x_n)$, $z_i=\frac1{\sqrt2}(x_i+\sqrt{-1}y_i)$. By \eqref{0825::eq::4}, $|Du|\geq c$ for some uniform $c$ on $\partial\Omega$. %Since $\partial\Omega$ is starshaped, 
Then we have
$$\sum_{l=1}^n(z_lu^{\varepsilon,r}_l+\bar z_lu^{\varepsilon,r}_{\bar l})=\sum_{l=1}^{2n}t_lu^{\varepsilon,r}_{t_l}=t\cdot\nu|Du^{\varepsilon,r}|\geq c\min_{\partial\Omega}t\cdot\nu:=c_1> 0.$$
Let $F^{i\bar j}=\frac{\partial}{\partial u^{\varepsilon,r}_{i\bar j}}(\log H_k[u^{\varepsilon,r}])$, $L=F^{i\bar j}\partial_{i\bar j}$.  Consider the function 
$$G :=2\mathrm{Re}\{z_lu^{\varepsilon,r}_l\}+Au^{\varepsilon,r}-B|z|^2,$$ where $A,B$ are constants to be determined later.
By calculation, we have
\begin{align*}
\mathcal F:=&\sum_{l=1}^nF^{i\bar i}=\sum_{l=1}^n\frac{S^{l\bar l}_k(\{u^{\varepsilon,r}_{i\bar j}\}_{1\leq i,j\leq n})}{S_k(\{u^{\varepsilon,r}_{i\bar j}\}_{1\leq i,j\leq n})}=(n-k+1)\frac{S_{k-1}(\{u^{\varepsilon,r}_{i\bar j}\}_{1\leq i,j\leq n})}{S_k(\{u^{\varepsilon,r}_{i\bar j}\}_{1\leq i,j\leq n})}\\\geq& k(C_n^k)^\frac1kS^{-\frac1k}_k(\{u^{\varepsilon,r}_{i\bar j}\}_{1\leq i,j\leq n})= k(C_n^k)^\frac1k\varepsilon^{-\frac1k}.
\end{align*}

On $\partial B_r$, we have $Du^{\varepsilon,r}=|Du^{\varepsilon,r}|\nu=|Du^{\varepsilon,r}|\frac tr$. It follows by \eqref{eq::0922::1} that
$$t\cdot Du^{\varepsilon,r}=r|Du^{\varepsilon,r}|\geq c_2r^{2-\frac{2n}{k}}\quad\text{on }\partial B_r.$$

By taking $r_4=\min\{4^\frac{2k}{2k-2n}R_0, (\frac{4a_0}{R_0^2})^{-\frac{k}{2n}}\}$, we have
$$\underline u\leq -\frac12r^{2-\frac{2n}k}\quad\text{on }\partial B_r.$$
It follows that if we take $A\leq \min\{\frac {c_1}2,c_2\}$, $B\leq \frac{c_1}{2R_0^2}$, $\varepsilon<\min\{\epsilon_1,\epsilon_2\}$, $r\leq \min\{r_3,r_4,r_5\}$, where $\epsilon_2:=\frac{C_n^kB^k}{(2+A)^k}$, $r_5= (\frac{c_2}{2B})^\frac{k}{2n}$,  then there holds 
$$G\geq c_1-A-BR_0^2\geq0\quad\text{on }\partial\Omega,$$
$$\quad G\geq (c_2-\frac A2)r^{2-\frac{2n}k}-Br^2\geq 0\quad\text{on }\partial B_r.$$
and
$$LG=(2+A)k-B\mathcal F=(2+A)k-Bk(C_n^k)^\frac1k\varepsilon^{-\frac1k}<0\quad\text{in }\Omega_r.$$
By maximum principle,
$$G\geq \min_{\partial \Omega_r}G>0.$$
Thus we prove $G>0$ in $\overline \Omega_r$ and \eqref{eq::0101::2} is obtained.

\subsection{Second order estimates}
Base on the key estimate \eqref{eq::est::2nd}, we can prove the global second order estimate in this subsection.

\subsubsection{\bf The global second order estimates can be reduced to the boundary second order estimates}
By Theorem \ref{thm::2ndest}, we have
$$\max_{\overline {\Omega_r}}H=\max_{\partial\Omega_r}H+C.$$
So 
\begin{align}\label{0814::eq::4}
u^{\varepsilon,r}_{\xi\bar \xi}(-u)^{-\frac{n}{n-k}}\leq C(\max_{\partial\Omega_r}H+C)\leq C(\max_{\partial\Omega_r}|\partial\bar \partial u^{\varepsilon,r}(-u^{\varepsilon,r})^{-\frac{n}{n-k}}|+1). 
\end{align}

On the other hand, let $D_\tau=\sum\limits_{i=1}^{2n}a_i\frac{\partial}{\partial t_i}$, with $\sum\limits_{i=1}^{2n}a_i^2=1$, from $Lu^{\varepsilon,r}_{\tau\tau}\geq 0$, we obtain 
$$u^{\varepsilon,r}_{\tau\tau}\leq \max_{\partial\Omega_r}|D^2u^{\varepsilon,r}|\quad\text{in }\Omega_r.$$
Since $u$ is subharmonic, we have 
$$-(2n-1)\max_{\partial\Omega_r}|D^2u^{\varepsilon,r}|\leq u_{t_it_i}\leq \max_{\partial\Omega_r}|D^2u^{\varepsilon,r}|\quad\text{in }\Omega_r.$$ 
Take $\tau=\frac1{\sqrt 2}(\frac{\partial}{\partial t_i}\pm\frac{\partial}{\partial t_j})$, we get
$$|u^{\varepsilon,r}_{t_it_j}|\leq C\max_{\partial\Omega_r}|D^2u^{\varepsilon,r}|\quad\text{in }\Omega_r.$$ 
Hence
$$|D^2u^{\varepsilon,r}|\leq C\max_{\partial\Omega_r}|D^2u^{\varepsilon,r}|\quad\text{in }\Omega_r.$$

\subsubsection{\bf Second order estimates on the boundary $\partial\Omega_r$}\ 
The second order estimate on $\partial\Omega$ is almost the same as in \cite{GMZ2022}. So we only need to prove the second order estimate on $\partial B_r$.

\noindent{\bf{Step 1. Pure tangential derivatives estimates}}

%Consider a boundary point on $\partial\Omega$. We may assume it to be the orgin of $\mathbb C^n$. Choose the coordinate $z_1,\cdots, z_n$ such that the $x_n$ axis is the inner normal direction to $\partial\Omega$ at $0$. Suppose that 
%$$t_1=y_1,\cdots, t_n=y_n, t_{n+1}=x_1,\cdots,t_{2n}=x_n.$$
%Denote by $t'=(t_1,\cdots,t_{2n-1})$. Then near the orgin, $\partial\Omega $ can be represented as a graph
%$$t_{2n}=x_n=\varphi(t')=B_{\alpha\beta}t_\alpha t_\beta+O(|t'|^3).$$
%We have
%on $\partial\Omega$, 
%$$u^{\varepsilon,r}(t',\varphi(t'))=0\quad\text{and}\quad\underline u(t'\varphi(t'))=0.$$
%Then
%$$u^{\varepsilon,r}_{t_it_j}=-u^{\varepsilon,r}_{x_n}\varphi_{t_it_j}\quad\text{and}\quad\underline u_{t_it_j}=-\underline u_{x_n}\varphi_{t_it_j}\quad\text{on }\partial\Omega.$$
%It follows that 
%$$u^{\varepsilon,r}_{t_it_j}=\frac{u^{\varepsilon,r}_{x_n}}{\underline u_{x_n}}\underline u_{t_it_j}\quad\text{on }\partial\Omega.$$
%By \eqref{0813::eq::1} and \eqref{0813::eq::2}, 
%$$C^{-1}\leq\frac{u^{\varepsilon,r}_{x_n}}{\underline u_{x_n}}\leq C.$$
%So
%$$u^{\varepsilon,r}_{t_it_j}\leq C\quad\text{on }\partial\Omega,$$
%and 
%\begin{align}\label{0814::eq::1}
%u^{\varepsilon,r}_{i\bar j}\geq C^{-1}\underline u_{i\bar j}\quad\text{on }\partial\Omega.
%\end{align}

Near $p\in \partial B_r$, we may assume $p=(0,\cdots,0,r)$. Near $\tilde p=(0,\cdots,0,1)$, $\partial B_1$ can be represented as a graph  
$$x_n=\rho(t')=\bigg(1-\sum_{i=1}^{2n-1}t_i^2\bigg)^{\frac12},$$
where 
%$t_i=y_i$ for $i=1,\cdots, n$ and $t_{i+n}=x_i$ for $i=1,\cdots,n-1$, 
$t'=(1,\cdots, t_{2n-1})$. %Set $t_{2n}=x_n$.

Let $\tilde u$ and $\tilde{\underline u}$ be the functions defined in \eqref{eq::0101::3} and \eqref{eq::0101::4}. Since $\tilde u$ is equal to some constant on $\partial B_1$, we have
$$\tilde u_{t_it_j}(\tilde p)=\tilde u_{x_n}(\tilde p)\delta_{ij}.$$
%$$\tilde u_{t_it_j}=\tilde{\underline u}_{t_jt_j}+(\tilde u-\tilde {\underline u})_{x_n}\delta_{ij}.$$
It follows 
\begin{align*}
|\tilde u_{t_it_j}(\tilde p)|\leq C.
\end{align*}
Hence
\begin{align*}\quad |u^{\varepsilon,r}_{t_it_j}(p)|\leq Cr^{-\frac{2n}k}.
\end{align*}
Furthermore, we have
%\begin{align}\label{0814::eq::2}
%\tilde u_{i\bar j}=\tilde{\underline u}_{i\bar j}+\frac12(\tilde u-\tilde{\underline u})_{x_n}\delta_{ij}.
%\end{align}
\begin{align}\label{0814::eq::2}
\tilde u_{i\bar j}(\tilde p)=\tilde u_{x_n}(\tilde p)\delta_{ij}.
\end{align}

\noindent{\bf{Step 2. Tangential-normal derivatives estimates}}

To estimate the tangential-normal second order derivatives on $\partial B_r$, we just estimate $\tilde u_{t_\alpha x_n}(\tilde p)$ for $\alpha=1,\cdots,2n-1$.
Note that $F^{i\bar j}$ and $\tilde u_{i\bar j}$ are both Hermitian matrix, and can be diagonalized by a same unitrary matrix, $F^{i\bar k}u_{j\bar k}$ is also an Hermitian matrix. It follows that 
$$F^{i\bar j}(z_r\tilde u_s-\bar z_s\tilde u_{\bar r})_{i\bar j}=0.$$
Now we estimate the mixed tangential-normal derivative $\tilde u_{t_ix_n}(\tilde p)$ for $\tilde p\in\partial B_1$.  Since $\tilde u(t',\rho(t'))$ is constant on $\partial B_1(0)$, we have
$$0=\tilde u_{t_\alpha}+\tilde u_{t_{2n}}\rho_{t_\alpha}=\tilde u_{t_\alpha}-\frac{t_\alpha}{\rho}\tilde u_{t_{2n}}, \quad\alpha=1,\cdots,2n-1.$$
That is on $\partial B_1\cap B_\frac12(\tilde p)$, 
$$ x_n\tilde u_{x_i}-x_i\tilde u_{x_n}=0\quad i=1,\cdots,n-1\quad\text{and}\quad x_n\tilde u_{y_i}-y_i\tilde u_{x_n}=0\quad i=1\cdots,n.$$ 
It follows that
$$y_n\tilde u_{x_i}-x_i\tilde u_{y_n}=0\quad i=1,\cdots, n-1\quad\text{and}\quad y_n\tilde u_{y_i}-y_i\tilde u_{y_n}=0\quad i=1,\cdots,n.$$

To estimate $\tilde u_{x_ix_n}(\tilde p)$ for $i=1,\cdots,n-1$, set 
%$$\begin{aligned}
%g=&2\mathrm{Re}\Big(z_{i-n}(\tilde u-\tilde{\underline u})_n-\bar z_n(\tilde u-\tilde{\underline u})_{\overline{i-n}}\Big)\\
%=&x_i(\tilde u-\tilde{\underline u})_{x_n}-x_n(\tilde u-\tilde{\underline u})_{x_i}+y_i(\tilde u-\tilde{\underline u})_{x_n}-y_n(\tilde u-\tilde{\underline u})_{x_i}.
%\end{aligned}$$
$$\begin{aligned}
g^1=2\mathrm{Re}(z_{i}\tilde u_n-\bar z_n\tilde u_{\overline{i}})=x_i\tilde u_{x_n}-x_n\tilde u_{x_i}+y_i\tilde u_{y_n}-y_n\tilde u_{y_i}.
\end{aligned}$$
Note that 
\begin{align*}
F^{i\bar j}g_{i\bar j}=&F^{i\bar j}(z_{i}\tilde u_n-\bar z_n\tilde u_{\overline{i}})_{i\bar j}+F^{i\bar j}(\bar z_{i}\tilde u_{\bar n}-z_n\tilde u_{{i}})_{i\bar j}=0.
\end{align*}
%\begin{align*}
%F^{i\bar j}g_{i\bar j}=&F^{i\bar j}\Big(z_{i-n}(\tilde u-\tilde{\underline u})_n-\bar z_n(\tilde u-\tilde{\underline u})_{\overline{i-n}}\Big)_{i\bar j}+F^{i\bar j}\Big(\bar z_{i-n}(\tilde u-\tilde{\underline u})_{\bar n}-z_n(\tilde u-\tilde{\underline u})_{{i-n}}\Big)_{i\bar j}\\
%\leq&C\mathcal F.
%\end{align*}
On $\partial B_1(0)\cap B_\frac12(\tilde p)$, consider the barrier function 
$$\Phi=A(1-x_n)\pm g^1.$$ 
Since $g_1$ is bounded on $\partial B_1(0)\cap B_\frac12(\tilde p)$, $1-x_n$ is bounded from below on $\partial B_\frac12(\tilde p)\cap B_1(0)$, we can choose a postive $A$ such that $\Phi\geq 0$ on $\partial(\partial B_1(0)\cap B_\frac12(\tilde p))$.
%We choose $A_1\gg A_2\gg 1$ such that $\Phi\geq 0$ on $\partial\Omega_\delta$ and $L\Phi\leq 0$ in $\Omega_\delta$. Therefore $\Phi\geq 0$ in $\partial\Omega_\delta$. 
It follows 
$$|g^1_{x_n}(\tilde p)|\leq C.$$ 
However, at $\tilde p$, we have
$$g^1_{x_n}=-\tilde u_{x_i}-\tilde u_{x_ix_n}.$$ 
Thus $$\tilde u_{x_ix_n}(\tilde p)\leq C, \quad i=1,\cdots,n-1.$$

To estimate $\tilde u_{y_ix_n}(\tilde p)$ for $i=1,\cdots,n$, set
$$\begin{aligned}
g^2=&2\mathrm{Im}(z_{i}\tilde u_n-\bar z_n\tilde u_{\overline{i}})
=y_i\tilde u_{x_n}-x_n\tilde u_{y_i}+y_n\tilde u_{x_i}-x_i\tilde u_{y_n}.
\end{aligned}$$
%and
%$$g^3=y_n\tilde u_{x_n}-x_n\tilde u_{y_n}.$$ After a similar proceeding, we obtain
%$$\begin{aligned}
%g=&2\mathrm{Im}\Big(z_{i}(\tilde u-\tilde{\underline u})_n-\bar z_n(\tilde u-\tilde{\underline u})_{\overline{i}}\Big)\\
%=&y_i(\tilde u-\tilde{\underline u})_{x_n}-x_n(\tilde u-\tilde{\underline u})_{y_i}+y_i(\tilde u-\tilde{\underline u})_{x_n}-x_n(\tilde u-\tilde{\underline u})_{y_i}.
%\end{aligned}$$
Proceeding similarly, we obtain
$$|\tilde u_{y_ix_n}(\tilde p)|\leq C,\quad i=1,\cdots,n.$$
%and
%$$|\tilde u_{y_nx_n}(\tilde p)|\leq C.$$
%Note that 
%$$F^{i\bar j}(y_n\tilde u_{x_n}-x_n\tilde u_{y_n})_{i\bar j}=0.$$
%To estimate $\tilde u_{y_nx_n}$, we consider

\noindent{\bf{Step 3. Double normal derivative estimate}}\ 

%By pure tangential derivative estimate on $\partial\Omega$, we have
%$$|u^{\varepsilon,r}_{y_ny_n}(0)|\leq C.$$ 
%To estimate $u^{\varepsilon,r}_{x_nx_n}$, it is suffices to estimate $u^{\varepsilon,r}_{n\bar n}$. By rotating $\{z_1,\cdots,z_{n-1}\}$, we may assume $\{u^{\varepsilon,r}_{i\bar j}\}_{1\leq i,j\leq n-1}$ is diagonal. Then
%$$\varepsilon=H_k[u^{\varepsilon,r}]=u^{\varepsilon,r}_{n\bar n}S_{k-1}(\{u^{\varepsilon,r}_{i\bar j}\}_{1\leq i,j\leq n-1})-\sum_{\beta=1}^{n-1}|u^{\varepsilon,r}_{\beta n}|^2S_{k-2}(\{u^{\varepsilon,r}_{i\bar j}\}_{1\leq i,j\leq n-1}).$$
%By \eqref{0814::eq::1}, we obtain
%$$S_{k-1}(\{u^{\varepsilon,r}_{i\bar j}\}_{1\leq i,j\leq n-1})\geq C^{1-k}S_{k-1}(\{\underline u_{i\bar j}\}_{1\leq i,j\leq n-1})\geq C^{1-k}C_n^{k-1}(C_n^k)^\frac{1-k}k\min_{\partial\Omega}H^{\frac{k-1}k}_k[\underline u]:= c_0.$$
%So 
%$$u^{\varepsilon,r}_{n\bar n}\leq \frac1{c_0}\Big(\varepsilon+\sum_{\beta=1}^{n-1}|u^{\varepsilon,r}_{\beta n}|^2S_{k-2}(\{u^{\varepsilon,r}_{i\bar j}\}_{1\leq i,j\leq n-1})\Big)\leq C.$$

By pure tangential derivative estimate on $\partial B_1$, we have
$$|\tilde u_{y_ny_n}(p)|\leq C.$$
To estimate $\tilde u_{x_nx_n}(\tilde p)$, it is suffices to estimate $\tilde u_{n\bar n}(\tilde p)$. By rotating $\{z_1,\cdots,z_{n-1}\}$, we may assume $\{\tilde u_{i\bar j}(\tilde p)\}_{1\leq i,j\leq n-1}$ is diagonal. Then
$$r^{2n}\varepsilon=H_k[\tilde u]=\tilde u_{n\bar n}S_{k-1}(\{\tilde u_{i\bar j}\}_{1\leq i,j\leq n-1})-\sum_{\beta=1}^{n-1}|\tilde u_{\beta n}|^2S_{k-2}(\{\tilde u_{i\bar j}\}_{1\leq i,j\leq n-1}).$$
By \eqref{0814::eq::2}, we obtain
\begin{align*}
S_{k-1}(\{\tilde u_{i\bar j}\}_{1\leq i,j\leq n-1})=& S_{k-1}(\{\tilde{ \underline u}_{i\bar j}\}_{1\leq i,j\leq n-1}+\frac12(\tilde u-\tilde{\underline u})_{x_n}I_{n-1})\\
\geq& S_{k-1}(\{\tilde{ \underline u}_{i\bar j}\}_{1\leq i,j\leq n-1})\\
\geq& C_n^{k-1}(C_n^k)^\frac{1-k}k\min_{\partial\Omega}H^{\frac{k-1}k}_k[\tilde{\underline u}]:= c_1.
\end{align*}
So
$$\tilde u_{n\bar n}(p)\leq C.$$
Combining these three cases together, and noting that $\tilde u$ is sunharmonic, we obtain
$$|\partial\bar\partial \tilde u|\leq C\quad\text{on }\partial B_1.$$
Hence
$$|\partial\bar\partial u^{\varepsilon,r}|\leq Cr^{-\frac{2n}k}\quad\text{on }\partial B_r.$$
By \eqref{0814::eq::4} and $C^0$ estimate, we have
$$|\partial\bar\partial u^{\varepsilon,r}|\leq C|z|^{-\frac{2n}k}\quad\text{in }\Omega_r.$$

\section{Proof of Theorem \ref{mainthm}}
\subsection{Uniqueness}The uniqueness follows from comparison theorem  \ref{comparison0718}.

Let $u,v$ be two solutions to \eqref{maineq} in $\Omega\setminus\{0\}$. For any $z_0\in\Omega\setminus\{0\}$, we first show $u(z_0)\leq v(z_0)$. In fact, for any $t\in (0,1)$, since
$u-tv =-(1-t)|z|^{2-\frac{2n}k}+C,$ there exists $r$ sufficiently small such that $z_0\in\Omega\setminus \overline B_r$ and $u<tv$ on $\partial B_r.$ Note that $u=-1<-t=tv$ on $\partial\Omega.$
By comparison theorem \ref{comparison0718}, we get
$u<tv$ in $\Omega\setminus B_r.$
Therefore
$u(z_0)\leq tv(z_0).$
Let $t\rightarrow 1$, we obtain
$u(z_0)\leq v(z_0).$
Hence 
$u\leq v$ in $\Omega\setminus\{0\}.$
Similarly, we obtain
$u\geq v$ in $\Omega\setminus\{0\}.$
Therefore $u=v$ in $\Omega\setminus\{0\}.$

\subsection{Existence}The existence follows from the uniform $C^2$-estimates for $u^{\varepsilon, r}$. 
%The proof is similar as that in \cite{gb2007imrn} by Guan.

%For any fixed $M_0>R_2$, for the solution to \eqref{approxeq}, by the $C^2$ estimates, we have????????

For $K=\Omega\setminus B_{r_0}$, %\subset\subset \Omega\setminus\{0\}$,  
for the solution to \eqref{approxeq}, by the estimate \eqref{eq::est::09051},
%$C^2$ estimates, 
we have
%\begin{align*}
%&-C-|z|^{2-\frac{2n}k}\leq u^{\varepsilon,r}\leq -|z|^{2-\frac{2n}k},\\
%&|Du^{\varepsilon,r}|\leq C|z|^{1-\frac{2n}k},\\
%&|\Delta u^{\varepsilon,r}|\leq C|z|^{-\frac{2n}k}.
%\end{align*}
%So 
\begin{align*}
|u^{\varepsilon,r}|_{C^1(K)}+|\Delta u^{\varepsilon,r}|\leq C(K).
\end{align*}
By Evans-Krylov theory, we obtain for any $0<\alpha<1$,
$$|u^{\varepsilon,r}|_{C^{2,\alpha}(K)}\leq C(K,\varepsilon).$$
%By Sobolev imbedding theorem, 
%$$|u^{\varepsilon,r}|_{C^{1,\alpha}(K)}\leq C(K).$$
%$$\|u^{\varepsilon,R}\|_{C^2(\overline \Sigma_{M_0})}\leq C_1\quad\text{ independent of }\varepsilon, R \text{ and }M_0.$$
%for all $R\geq M_0$. By the Evans-Krylov theory, we obtain for $0<\alpha<1$,
%$$\|u^{\varepsilon,R}\|_{C^{2,\alpha}(\overline \Sigma_{M_0})}\leq C_2(\varepsilon, M_0)\quad\text{ independent of }R .$$
By compactness, we can find a sequence $r_i\rightarrow0$ such that
$$u^{\varepsilon,r_i}\rightarrow u^\varepsilon\quad\text{in }C^{2}(K).$$
where $u^\varepsilon$ satisfies
\begin{equation*}\label{approxeq1}
\begin{cases}
H_k[u^\varepsilon]=\varepsilon&\quad\text{in }K,\\
u=-1&\quad\text{on }\partial\Omega,
\end{cases}
\end{equation*}
and 
\begin{equation}\label{0814::eq::5}\begin{aligned}
	&-C-|z|^{2-\frac{2n}k}\leq u^{\varepsilon}(z)\leq -|z|^{2-\frac{2n}k},\\
	&|Du^{\varepsilon}(z)|\leq C|z|^{1-\frac{2n}k},\\
	&|\p\bar \p u^{\varepsilon}(z)|\leq C|z|^{-\frac{2n}k}.
\end{aligned}\end{equation}
Moreover,
\begin{equation*}
	|u^{\varepsilon}|_{C^{2,\alpha}(K)}\leq C(K,\varepsilon).
	\end{equation*}
%which satisfies the desired estimates.
By the classical Schauder theory,  $u^{\varepsilon}$ is smooth.

By above estimates \eqref{0814::eq::5} for $u^{\varepsilon}$, for any sequence $\varepsilon_j\rightarrow 0$, there is a subsequence of $\{u^{\varepsilon_j}\}$ converging to a function $u$ in $C^{1,\alpha}$ norm on any compact subset of $\Omega\setminus \{0\}$. 
Thus $u\in C^{1,\alpha}(\Omega\setminus \{0\})$ and satisfies the estimates \eqref{mainest1} and \eqref{mainest2}. By the convergence theorem of the complex $k$-Hessian operator proved by Trudinger-Zhang \cite{TrudingerZhang2014} (see also Lu \cite{Lu2015}), $u$ is a  solution to \eqref{maineq}.

\section*{Acknowledgements:}The second author was supported by National Natural Science Foundation of China (grants 11721101 and 12141105) and National Key Research and Development Project (grants SQ2020YFA070080). The third author was supported by NSFC grant No. 11901102.

\bibliographystyle{plain}
\bibliography{2023-04-16-Gao-Ma-Zhang}

\end{document}